\documentclass[reqno]{amsart}
\usepackage{amsmath}
\usepackage{amssymb}
\usepackage{amsthm,hyperref}
\usepackage[all]{xy}
\usepackage[utf8]{inputenc}
\usepackage{enumerate}

\renewcommand{\hom}{\operatorname{Hom}}
\newcommand{\tor}{\operatorname{Tor}}
\newcommand{\ext}{\operatorname{Ext}}
\newcommand{\spec}{\operatorname{Spec}}
\newcommand{\height}{\operatorname{ht}}
\newcommand{\pd}{\operatorname{pd}}
\newcommand{\fd}{\operatorname{fd}}

\newcommand{\gid}{\operatorname{Gid}}
\newcommand{\p}{\mathfrak{p}}
\newcommand{\q}{\mathfrak{q}}
\newcommand{\m}{\mathfrak{m}}
\newcommand{\bd}{\mathcal{B}_D}

\renewcommand{\subset}{\subseteq}

\newcommand{\catb}[2]{\mathcal{B}_{#1}(#2)}
\newcommand{\ol}{\overline}

\newcommand{\supp}{\operatorname{supp}}
\newcommand{\kp}[1][p]{\kappa(\mathfrak{#1})}
\newcommand{\tp}[1][p]{t(\mathfrak{#1})}
\newcommand{\Tor}[4][R]{\operatorname{Tor}^{#1}_{#2}(#3,#4)}
\newcommand{\Ext}[4][R]{\operatorname{Ext}_{#1}^{#2}(#3,#4)}	
\newcommand{\ici}{\mathcal{I}_C\mathcal{I}}
\newcommand{\bc}{\mathcal{B}_C}
\renewcommand{\oplus}{\bigoplus}

\newtheorem{theorem}{Theorem}[section]
\newtheorem{intro}{Theorem}

\newtheorem{lemma}[theorem]{Lemma}
\newtheorem{proposition}[theorem]{Proposition}
\newtheorem{corollary}[theorem]{Corollary}
\theoremstyle{definition}
\newtheorem{definition}[theorem]{Definition}
\newtheorem{example}[theorem]{Example}
\newtheorem{rmk}[theorem]{Remark}
\newtheorem{fact}[theorem]{Fact}
\newtheorem{notn}[theorem]{Notation}
\newtheorem{assumption}[theorem]{Assumption}
\newtheorem{question}[theorem]{Question}

\numberwithin{equation}{theorem}

\title[Gorenstein injective filtrations]{Gorenstein injective filtrations over Cohen-Macaulay rings with dualizing modules}
\author{Aaron J. Feickert}
\author{Sean Sather-Wagstaff}
\address{Dept.\ of Mathematics, NDSU Dept \#2750, PO Box 6050, Fargo ND 58108-6050 USA}
\email{aaron.feickert@gmail.com}
\email{sean.sather-wagstaff@ndsu.edu}
\urladdr{http://www.ndsu.edu/pubweb/\~{}ssatherw/}
\thanks{Sean Sather-Wagstaff was supported in part by a grant from the NSA}


\begin{document}

\maketitle

\begin{abstract}
Over a noetherian ring, it is a classic result of Matlis that injective modules admit direct sum decompositions into injective hulls of quotients by prime ideals. We show that over a Cohen-Macaulay ring admitting a dualizing module, Gorenstein injective modules admit similar filtrations. We also investigate Tor-modules of Gorenstein injective modules  over such rings. This extends work of Enochs and Huang over Gorenstein rings. 

Furthermore, we give examples showing the following: (1) the class of Gorenstein injective $R$-modules need not be closed under tensor products, even when $R$ is local and artinian; (2) the class of Gorenstein injective $R$-modules need not be closed under torsion products, even when $R$ is a local, complete hypersurface; and (3) the filtrations given in our main theorem do not yield direct sum decompositions, even when $R$ is a local, complete hypersurface.
\end{abstract}

\section{Introduction}
Throughout this paper, let $R$ be a commutative noetherian ring with identity. 

\

In the classic paper~\cite{matlis}, Matlis shows that injective modules over noetherian rings have direct sum decompositions into injective hulls of the form $E_R(R/\p)$, where $\p$ is prime. Our goal is to prove a similar result for Gorenstein injective modules.
Recall that an $R$-module $G$ is {\it Gorenstein injective} (or \emph{G-injective}, for short) if it is the image of some map in a complete injective resolution; that is, if there exists an exact sequence 
of injective $R$-modules
$$E=\cdots \to E_1 \to E_0 \to E^0 \to E^1\to \cdots$$ such that $\hom_R(E,I)$ is exact for all injective $R$-modules $I$ and $G \cong \operatorname{Im}(E_0 \to E^0)$.

In~\cite{enochs}, Enochs and Huang make the first progress on the above-stated goal. They
prove that over a Gorenstein ring of finite Krull dimension, Gorenstein injective modules admit filtrations 
such that subsequent quotients decompose as direct sums indexed by prime ideals of fixed height. They use this result to show that the class of Gorenstein injective modules is closed under tensor products over such rings.

In this paper, we answer the question posed in~\cite[Remark 3.2]{enochs} and extend the results of Enochs and Huang to Cohen-Macaulay rings admitting a {\it dualizing module}; that is, a finitely-generated $R$-module $D$ of finite injective dimension such that the homothety map $\chi_D^R: R \to \hom_R(D,D)$ is an isomorphism and we have $\ext_R^i(D,D) = 0$ for all $i \geq 1$; see Fact~\ref{fact141223a}.
Specifically, we prove the following result in Theorem~\ref{20141010.4} below.

\begin{intro}
\label{20141114.1}
Let $R$ be a $d$-dimensional Cohen-Macaulay ring with a dualizing module $D$. If $G$ is a Gorenstein injective $R$-module, then $G$ has a filtration $$0 = G_{d+1} \subset G_d \subset \cdots \subset G_1 \subset G_0 = G$$ such that each 
submodule $G_k$ and each quotient 
$G_k/G_{k+1}\cong\oplus_{\height(\p)=k}G_{(\p)}$
is Gorenstein injective and 
each module 
$G_{(\p)}:=\Tor k{\hom_R(D,E_R(R/\p))}{G}$
is Gorenstein injective and 
satisfies $\tp$; see Definition~\ref{def:tp}.
Moreover, this filtration and the direct sum compositions of the factors are unique and functorial.
\end{intro}

From this, one might expect us to follow Enochs and Huang's lead by proving that the class of Gorenstein injective modules is closed under tensor products in our setting.
However, Example~\ref{ex141218a} below shows that this is not the case for non-Gorenstein rings.
Furthermore, Example~\ref{ex141231a} 
shows that the class of Gorenstein injective modules is not closed under Tor-modules, even when 
$R$ is a local, complete  hypersurface, addressing~\cite[Remark 4.2]{enochs}.
In addition, we show in Example~\ref{ex150430a} that the filtration from Theorem~\ref{20141114.1} does not give a direct sum decomposition, as one might expect
given Matlis' result, even when $R$ is a local, complete  hypersurface.

In contrast with Example~\ref{ex141218a}, though, we prove the following result in Theorem~\ref{cor150430a}.
It says 
that, under suitable hypotheses, the class of Gorenstein injective modules is closed
under tensor products.

\begin{intro}
\label{20141114.2}
Let $R$ be Cohen-Macaulay with a dualizing module $D$. 
Assume that $R$ is generically Gorenstein.
If $G$ and $H$ are Gorenstein injective $R$-modules, then $G \otimes_R H$ is also Gorenstein injective.
\end{intro}

We conclude this introduction by summarizing the organization of this paper.
Section~\ref{sec150101a} contains background information, and
Section~\ref{sec150101b} consists of technical results central to our main proofs.
Section~\ref{sec141229a} contains the proof of  Theorem~\ref{20141114.1}, and come consequences.
Section~\ref{sec150430a} consists of the aforementioned examples.
The paper concludes with Section~\ref{sec141229b} which gives a more general version of  Theorem~\ref{20141114.1},
in addition to the proof of Theorem~\ref{20141114.2}.

\section{Foundational Notions}
\label{sec150101a}

We begin with a definition due to Foxby~\cite{foxby:gmarm}, generalizing Grothendieck's 
notion of a dualizing module from~\cite{hartshorne:lc}, and introduced independently by Golod~\cite{golod:gdagpi} and Vasconcelos~\cite{vasconcelos:dtmc}.

\begin{definition}
A finitely-generated $R$-module $C$ is {\it semidualizing} if the homothety map $\chi_C^R: R \to \hom_R(C,C)$ is an isomorphism, and $\ext_R^i(C,C) = 0$ for all $i \geq 1$.
\end{definition}

\begin{assumption}
We assume for the remainder of this section that $C$ is a semidualizing $R$-module.
\end{assumption}

\begin{fact}\label{fact141223a}
The ring $R$ always has a semidualizing module, namely, the $R$-module $R$. 
By definition, a dualizing module is a semidualizing module of finite injective dimension.
Not every ring has a dualizing module: $R$ has a dualizing module if and only if it is Cohen-Macaulay and a homomorphic image of a Gorentein ring with finite Krull dimension; see~\cite{foxby:gmarm,reiten:ctsgm,sharp:gmccmlr}. The proof of one implication uses Nagata's ``idealization'' of $D$ (a.k.a., the ``trivial extension'' of $R$ by $D$). As we use this construction in the sequel, we describe it here.

Let $M$ be a finitely-generated $R$-module. Endow the direct sum $R\oplus M$ with the following binary product: $(r,m)(s,n):=(rs,rn+sm)$. This makes $R\oplus M$ into a commutative noetherian ring with identity $(1,0)$, which we denote $R\ltimes M$. Note that the natural epimorphism $R\ltimes M\to R$ is a ring homomorphism. In particular, every $R$ module has the structure of an $R\ltimes M$-module. It is shown in~\cite{foxby:gmarm,reiten:ctsgm} that if $D$ is dualizing for $R$, then $R\ltimes D$ is Gorenstein.

Note that if $C$ is semidualizing for $R$ and $D$ is dualizing for $R$, then $\hom_R(C,D)$ is semidualizing for $R$; see, e.g.,~\cite[4.11]{vasconcelos:dtmc}.
\end{fact}

Next, we have the central objects of study in this paper. In the case $C=R$, they were introduced by Enochs and Jenda~\cite{enochs:gipm}. The general definition is due to Holm and J\o rgensen~\cite{holm:smarghd}.

\begin{definition}\label{defn150101a}
A \emph{complete $\ici$-resolution} is an exact sequence $X$ of the form
$$\cdots\to\hom_R(C,I_1)\to\hom_R(C,I_0)\xrightarrow{\partial}I^0\to I^1\to\cdots$$
such that $I_j$ and $I^j$ are injective for all $j$, and such that $\hom_R(\hom_R(C,J),X)$ is exact for each injective $R$-module $J$. An $R$-module $G$ is \emph{$C$-Gorenstein injective} if it has a complete $\ici$-resolution; that is, if there is a complete $\ici$-resolution as above such that $G\cong\operatorname{Im}(\partial)$. 
Note that in the case $C=R$, these are the already-defined
complete injective resolution and Gorenstein injective module.

The \emph{$C$-Gorenstein injective dimension} of an $R$-module $M$, denoted $C\text{-}\gid_R(M)$, is the length $n$ of the shortest resolution
$$0\to M\to G^0\to\cdots\to G^n\to 0$$
of $M$ by $C$-Gorenstein injective modules, if such a bounded resolution exists. If no such resolution exists, then $C\text{-}\gid_R(M)=\infty$.
In the case $C=R$, we write $\gid_R(M)$ instead of $R\text{-}\gid_R(M)$.
\end{definition}

\begin{example}\label{ex141223a}
If $I$ is an injective $R$-module, then $I$ and $\hom_R(C,I)$ are $C$-Gorenstein injective over $R$.
\end{example}

The following fact shows the deep connection between the case $C=R$ and the general case.

\begin{fact}[\protect{\cite[2.13(1) and 2.16]{holm:smarghd}}]
\label{fact141223b}
An $R$-module $G$ is $C$-Gorenstein injective over $R$ if and only if it is Gorenstein injective over $R\ltimes C$; see Fact~\ref{fact141223a}. Moreover, one has $C\text{-}\gid_R(M)=\gid_{R\ltimes C}(M)$ for each $R$-module $M$.
(See also~\cite{aung}.)
\end{fact}

Next, we note some useful consequences of Fact~\ref{fact141223b} and~\cite{holm}.

\begin{lemma}\label{lem141223a}
Let $M$ be an $R$-module such that $C\text{-}\gid_R(M)<\infty$.
\begin{enumerate}[(a)]
\item \label{lem141223a1}
One has $C\text{-}\gid_R(M)\leq\dim(R)$.
\item \label{lem141223a2}
If there is an exact sequence
$$0\to N\to G_{\dim(R)}\to\cdots\to G_1\to M\to 0$$
such that each $G_i$ is $C$-Gorenstein injective over $R$, then $M$ is $C$-Gorenstein injective as well.
\end{enumerate}
\end{lemma}

\begin{proof}
\eqref{lem141223a1}
This is by Fact~\ref{fact141223b} and~\cite[Theorem 2.29]{holm}, as follows:
$$C\text{-}\gid_R(M)=\gid_{R\ltimes C}(M)\leq\dim(R\ltimes C)=\dim(R).$$
See also, e.g.,~\cite[1.4 and 3.3]{3danes}.

\eqref{lem141223a2}
From the given exact sequence, Fact~\ref{fact141223b} implies that $$\gid_{R\ltimes C}(N)=C\text{-}\gid_R(N)<\infty$$
so $\gid_{R\ltimes C}(N)\leq\dim(R)$ by part~\eqref{lem141223a1}. Using~\cite[Theorem 2.22]{holm}, we conclude that $M$ is Gorenstein injective over $R\ltimes C$, 
so $M$ is $C$-Gorenstein injective over $R$ by Fact~\ref{fact141223b}.
\end{proof}

The following class originates in~\cite{foxby:gmarm}. It is incredibly useful for studying the above homological dimensions, because of Fact~\ref{fact141223c}.

\begin{definition}
The {\it Bass class} $\bc(R)$ consists of all $R$-modules $M$ such that
the evaluation map $C \otimes_R \hom_R(C,M) \to M$ is an isomorphism, and
$\ext_R^i(C,M) = 0 = \tor_i^R(C,\hom_R(C,M))$ for all $i \geq 1$.
We write $\bc$ for this class if the ring $R$ is understood.
\end{definition}

\begin{fact}\label{fact141223c}
From~\cite[Corollary 6.3]{holm:fear}, we know that $\mathcal B_C$ has the ``two-of-three" property for short exact sequences of $R$-modules; that is, given
an exact sequence $0\to M_1\to M_2\to M_3\to 0$ of $R$-modules, if two of the $M_i$ are in $\bc$, then so is the third $M_i$.

Assume that $R$ has a dualizing module $D$. Then an $R$-module $M$ has finite $C$-Gorenstein injective dimension if and only if it is in $\mathcal{B}_{\hom_R(C,D)}$ by~\cite[Theorem 4.6(2)]{holm:smarghd}. In particular, one has $\gid_R(M)<\infty$ if and only if $M\in\bd$; see also~\cite[Theorem 4.1]{3danes}. 
\end{fact}

The next notion is convenient for this investigation.

\begin{definition}
\label{def:tp}
Let $\p \in \spec(R)$. We say an $R$-module $S$ has property $t(\p)$ if
\begin{enumerate}[(a)]
\item for each $r \in R\setminus\p$ the map $S \xrightarrow{r} S$ is an isomorphism; and
\item for each $x \in S$ we have $\p^mx=0$ for some $m \geq 1$.
(In this case, we say that $x$ is \emph{locally nilpotent} on $S$.)
\end{enumerate}
In particular, the $R$-module $\kp:=R_{\p}/\p R_{\p}$ 
and the injective hull $E_R(R/\p)$ both satisfy $t(\p)$; see~\cite[Lemma 3.2]{matlis}.
\end{definition}

We continue with a brief discussion of local cohomology~\cite{hartshorne:lc}.

\begin{definition}\label{notn141231a}
Let $\mathfrak a$ be an ideal of $R$, and let $N$ be an $R$-module.
The \emph{$\mathfrak a$-torsion submodule of $N$} is 
defined as $\Gamma_{\mathfrak a}(N):=\{n\in N\mid\text{$\mathfrak a^mn=0$ for $m\gg 0$}\}$.
The \emph{$i$th local cohomology module} of $N$ is $H^i_{\mathfrak a}(N):=H^i(\Gamma_{\mathfrak a}(J))$ where $J$ is an injective resolution of $N$.
\end{definition}

\begin{fact}\label{fact141231a}
Let $\mathfrak a$ be an ideal of $R$, and let $N$ be an $R$-module.
It is straightforward to show that the operation $\Gamma_{\mathfrak a}(-)$ is a left-exact covariant functor.
If an $R$-module $N$ satisfies $\tp[q]$ for some $\q$, then one has
$$\Gamma_{\p}(N_{\p})\cong
\begin{cases}
0&\text{if $\p\neq \q$} \\
N&\text{if $\p=\q$.}
\end{cases}
$$
In particular, we have
$$\Gamma_{\p}(E_R(R/\q)_{\p})\cong
\begin{cases}
0&\text{if $\p\neq \q$} \\
E_R(R/\q)&\text{if $\p=\q$.}
\end{cases}
$$
\end{fact}

The last definition of this section is due to Foxby~\cite{foxby:bcfm}.

\begin{definition}\label{defn141224a}
The \emph{small support} of an $R$-module $M$ is
$$\supp_R(M):=\{\p\in\spec(R)\mid\text{$\Tor i{\kp}M\neq 0$ for some $i$}\}.$$
\end{definition}

\begin{fact}\label{fact141224a}
Let $M$ be an $R$-module with minimal injective resolution $J$. Then $\p\in\supp_R(M)$ if and only if
$E_R(R/\p)$ is a summand of $J^i$ for some $i$; see~\cite[Remark 2.9]{foxby:bcfm}. 
Also, we have $M\neq 0$ if and only if $\supp_R(M)\neq\emptyset$, by~\cite[Lemma 2.6]{foxby:bcfm}.
Given a second $R$-module $N$, we have $\Tor iMN\neq 0$ for some $i$ if and only if $\supp_R(M)\cap\supp_R(N)\neq \emptyset$,
by~\cite[Proposition 2.7]{foxby:bcfm}.
\end{fact}

The next  lemma is implicit in~\cite{enochs}.

\begin{lemma}\label{lem141224c}
Let $M$ and $N$ be  $R$-modules, and let $\p,\q\in\spec(R)$ such that $\p\neq\q$.
If $M$ satisfies $t(\p)$ and $N$ satisfies $t(\q)$, then $\Tor iMN=0$ for all $i$.
\end{lemma}

\begin{proof}
Assume without loss of generality that $\p\not\subseteq\q$, and let $x\in\p\smallsetminus\q$.
If follows that $x$ is locally nilpotent on $M$ and the map $N\xrightarrow xN$ is an isomorphism.
From this, it is straightforward to show that $x$ is locally nilpotent on $\Tor iMN$ and that
the map $\Tor iMN\xrightarrow{x^n}\Tor iMN$ is an isomorphism for each $n\geq 1$. 
It follows readily that $\Tor iMN=0$.
\end{proof}

\begin{rmk}\label{rmk141224b}
In the setting of Lemma~\ref{lem141224c}, we may not have $\Ext iMN=0$. 
Indeed, assume that $(R,\m,k)$ is local and non-artinian.
Then there is a prime $\p\neq\m$ in $\spec(R)$.
The modules $M=E_R(R/\p)$ and $N=E_R(R/\m)$ satisfy $t(\p)$ and $t(\m)$, respectively,
but we have $\hom_R(M,N)\neq 0$ since $N$ is faithfully injective and $M\neq 0$.
\end{rmk}

\begin{lemma}\label{lem141224a}
Let $M$ be a non-zero $R$-module, and let $\p\in\spec(R)$.
If $M$ satisfies $t(\p)$, then $\supp_R(M)=\{\p\}$.
\end{lemma}

\begin{proof}
Fact~\ref{fact141224a} implies that $\supp_R(M)\neq\emptyset$, so it suffices to show that $\supp_R(M)\subseteq\{\p\}$.
Let $\q\in\spec(R)$ such that $\p\neq \q$; it suffices to show that $\q\notin\supp_R(M)$; that is, that $\Tor i{\kp[q]}M=0$ for all $i$.
Since $\kp[q]$ satisfies $\tp[q]$, this follows from Lemma~\ref{lem141224c}.
\end{proof}

\section{Preliminary Results}
\label{sec150101b}

This section consists of 
useful results about tensor products and Tor-modules for the proofs of our main theorems.
(Note that special cases of some of these results can be found in~\cite{MR3169700}.)
For the first two results, recall the following.
For $i \geq 0$, the  ring $R$ is $(S_i)$ if $\operatorname{depth}(R_\p) \geq \min(\height(\p),i)$ for all $\p \in \spec(R)$. In particular, if $R$ is $(S_1)$, then the associated primes of $R$ are all minimal.
Since $R$ is Cohen-Macaulay if and only if it is $(S_i)$ for all $i \geq 0$, the next two results hold in particular if $R$ is Cohen-Macaulay. 

\begin{lemma}
\label{20141002.2}
Let $R$ be $(S_1)$. If $\p$ is a prime ideal of $R$ such that $\height(\p) \geq 1$ and $T$ is an $R$-module with property $t(\p)$, then for any Gorenstein injective $R$-module $G$ we have $G \otimes_R T = 0$.
\end{lemma}

\begin{proof}
Let $\p \in \spec(R)$ such that $\height(\p) \geq 1$. Since $R$ is $(S_1)$, we have $\operatorname{Ass}(R) = \operatorname{Min}(R)$. Prime avoidance provides an $R$-regular element $r \in \p$. Then $\pd_R(R/(r)) = 1$ with free resolution $$0 \to R \xrightarrow{r} R \to R/(r) \to 0.$$ By~\cite[Lemma 1.3]{balance} we have the vanishing $\ext_R^1(R/(r),G) = 0$, so the sequence $$\hom_R(R,G) \xrightarrow{r} \hom_R(R,G) \to 0$$ is exact, and hence the map $G \xrightarrow{r^n} G$ is a surjection for all $n \geq 1$.

To show that $G \otimes_R T = 0$, it suffices to show that for $x \in G$ and $y \in T$, one has $x \otimes y = 0$. Choose $n \geq 1$ such that $r^ny=0$. The previous paragraph gives an element $x' \in G$ such that $r^nx' = x$, so $$x \otimes y = (r^nx') \otimes y = x' \otimes (r^ny) = 0$$ as desired.
\end{proof}

\begin{lemma}
\label{20141002.3}
Let $R$ be $(S_1)$ and let $D$ be a finitely-generated $R$-module. If $\p \in \spec(R)$ such that $\height(\p) \geq 1$, then for any Gorenstein injective $R$-module $G$ we have $\hom_R(D,E_R(R/\p)) \otimes_R G = 0$.
\end{lemma}

\begin{proof}
Since $D$ is finitely generated and $E_R(R/\p)$ has property $t(\p)$, we conclude that $\hom_R(D,E_R(R/\p))$ also has property $t(\p)$. Thus, the desired vanishing follows from Lemma~\ref{20141002.2}.
\end{proof}

We use the next lemma to compute certain $\tor$-modules below.

\begin{lemma}
\label{20141030.1}
Let $R$ be a $d$-dimensional Cohen-Macaulay ring with a dualizing module $D$, and let $\p \in \spec(R)$. If $J$ is the minimal injective resolution of $D$ over $R$, then $\hom_R(J,E_R(R/\p))$ is a flat resolution (over $R$ and over $R_\p$) of $\hom_R(D,E_R(R/\p)) \cong \hom_{R_\p}(D_\p,E_{R_\p}(\kappa(\p)))$. Furthermore, 
there is an inequality $\fd_R(\hom_R(D,E_R(R/\p))) \leq \height(\p)$.
\end{lemma}

\begin{proof}
Consider the augmented minimal injective resolution 
$${^+}J := \quad 0 \to D \to \bigoplus_{\height(\q)=0} E(R/\q) \to \bigoplus_{\height(\q)=1} E(R/\q) \to \cdots \to \bigoplus_{\height(\q)=d} E(R/\q) \to 0$$ 
of $D$. Since $\hom_R(E_R(R/\p),E_R(R/\q)) \neq 0$ if and only if $\p \subset \q$ by~\cite[Theorem 3.3.8(5)]{rha} and $\q \not\subset \p$ for $\height(\q) > \height(\p)$, we have the following exact sequence:
\begin{multline*}
\hom({^+}J,E(R/\p)) = \quad 0 \to \prod_{\substack{\height(\q)=\\\height(\p)}}\hom(E(R/\q),E(R/\p)) \to \cdots \\ 
\cdots \to \prod_{\height(\q)=0}\hom(E(R/\q),E(R/\p)) \to \hom(D,E(R/\p)) \to 0
\end{multline*}
Since $R$ is noetherian and each $\hom_R(E_R(R/\q),E_R(R/\p))$ is flat, we conclude that $\hom_R(J,E_R(R/\p))$ is a flat resolution of $\hom_R(D,E_R(R/\p))$. It follows immediately that $\fd_R(\hom_R(D,E_R(R/\p))) \leq \height(\p)$.

For any $\q \in \spec(R)$ we have the isomorphisms
\begin{eqnarray*}
\hom_R(E_R(R/\q),E_R(R/\p)) &\cong& \hom_R(E_R(R/\q),\hom_{R_\p}(R_\p,E_R(R/\p))) \\
&\cong& \hom_{R_\p}(R_\p \otimes_R E_R(R/\q),E_R(R/\p)) \\
&\cong& \hom_{R_\p}(E_R(R/\q)_\p,E_{R_\p}(\kappa(\p)))
\end{eqnarray*}
that follow from standard localization properties and the isomorphism $E_R(R/\p) \cong E_{R_\p}(\kappa(\p))$, as in~\cite[Theorem 18.4]{matsumura}. As $E_R(R/\q)$ is injective over $R$, the module $E_R(R/\q)_\p$ is injective over $R_\p$ by~\cite[Proposition 3.3.2]{rha}. We conclude that $\hom_{R_\p}(E_R(R/\q)_\p,E_{R_\p}(\kappa(\p)))$ is flat over $R$ and over $R_\p$. Similarly, one has
$$\hom_R(D,E_R(R/\p)) \cong \hom_{R_\p}(D_\p,E_{R_\p}(\kappa(\p)))$$ 
and $D_\p$ is a dualizing $R_\p$-module with minimal injective resolution $J_\p$.
Thus, we have $\hom_R(J,E_R(R/\p)) \cong \hom_{R_\p}(J_\p,E_{R_\p}(\kappa(\p)))$, 
and this complex is a flat resolution  of $\hom_{R_\p}(D_\p,E_{R_\p}(\kappa(\p)))$ over $R_\p$.
\end{proof}

The next few results give  useful  descriptions of certain Tor-modules.

\begin{proposition}
\label{prop141231b}
Let $R$ be Cohen-Macaulay ring with a dualizing module $D$, and let $\p\in \spec(R)$ with $h=\height(\p)$. Given an $R$-module $M$, for all $k$ we have
$$\tor_k^R(\hom_R(D,E_R(R/\p)),M) \cong H^{h-k}_\p(M_\p)\cong H^{h-k}_{\p R_{\p}}(M_{\p}).$$ 
\end{proposition}

\begin{proof}
Assume first that $R$ is local with maximal ideal $\m=\p$, and set $E=E_R(R/\m)$.
Let $x_1,\ldots,x_h\in\m$ be a system of parameters for $R$.
Recall that the ``\v{C}ech complex''
$$C=\quad 0\to R\to\bigoplus_iR_{x_i}\to\bigoplus_{i<j}R_{x_ix_j}\to\cdots\to R_{x_1\cdots x_h}\to 0$$
concentrated in cohomological degrees 0 to $h$, computes the local cohomology of $R$; see, e.g.~\cite[Section 3.5]{bruns:cmr}.
Grothendieck's Local Duality Theorem~\cite[Theorem 3.5.8]{bruns:cmr} implies that $D^\vee\cong H^h_{\m}(R)$ and $H^i_\m(R)=0$ for all $i\neq d$.
In other words, by shifting $C$ to \emph{homological} degrees 0 to $h$, we obtain a flat resolution $F$ of $D^\vee$.
It follows that
$\tor_k^R(\hom_R(D,E),M)
\cong H_k(F\otimes_RM)$.
Note that we have
\begin{align*}
F\otimes_RM
&\cong\quad 0\to M\to\bigoplus_iM_{x_i}\to\bigoplus_{i<j}M_{x_ix_j}\to\cdots\to M_{x_1\cdots x_d}\to 0.
\end{align*}
We conclude that
\begin{align*}
\tor_k^R(\hom(D,E),M)
&\cong H_k(F\otimes_RM)\\
&\cong H_{k-h}(C\otimes_RM)\\
&\cong H^{h-k}(C\otimes_RM)\\
&\cong H^{h-k}_\m(M)
\end{align*} 
as desired.

For the general case, let $J$ be an injective resolution of $D$ over $R$.
Then we have
\begin{align*}
\tor_k^R(\hom_R(D,E_R(R/\p)),M) 
&\cong H_k(\hom(J,E_R(R/\p)) \otimes_R M) \\
&\cong H_k(\hom_R(J,E_R(R/\p)) \otimes_{R_\p} M_\p) \\
&\cong \tor_k^{R_\p}(\hom_R(D,E_R(R/\p)),M_\p) \\
&\cong \tor_k^{R_\p}(\hom_{R_\p}(D_\p,E_{R_\p}(\kappa(\p))),M_\p)\\
&\cong H^{h-k}_\p(M_\p)\\
&\cong H^{h-k}_{\p R_{\p}}(M_{\p})
\end{align*}
where the fifth isomorphism is from the local case, and the last isomorphism is standard.
\end{proof}

In light of Proposition~\ref{prop141231b}, the next few results also yield isomorphisms and vanishing  for modules of the form
$H^i_{\p}(M_{\p})$ and $H^i_{\p R_{\p}}(M_{\p})$.

\begin{proposition}
\label{prop141231a}
Let $R$ be Cohen-Macaulay with a dualizing module $D$, and let $\p,\q \in \spec(R)$. Then we have
$$\tor_k^R(\hom_R(D,E_R(R/\p)),E_R(R/\q)) \cong
\begin{cases}0&\text{if $\p\neq \q$ or if $k\neq \height(\p)$}\\ E_R(R/\p)&\text{if $\p= \q$ and $k= \height(\p)$.}\end{cases}
$$ 
\end{proposition}

\begin{proof}
Proposition~\ref{prop141231b} shows that we may assume without loss of generality that
$R$ is local and $\p = \m$ is the unique maximal ideal. Set $d:=\dim(R)=\height(\p)$.

If $\m \neq \q$, then the desired vanishing is from Lemma~\ref{lem141224c}, since $E_R(R/\q)$ satisfies $\tp[q]$
and $\hom_R(D,E_R(R/\p))$ satisfies $\tp$.
Thus, we assume that $\p = \q = \m$, and set $E:=E_R(R/\m)$.
From Proposition~\ref{prop141231b}, 
we have the first isomorphism in the next sequence.
$$\tor_k^R(\hom_R(D,E),E)
\cong H^{d-k}_\m(E)
\cong
\begin{cases}
0&\text{if $d-k\neq 0$, i.e., $k\neq\height(\p)$}\\
E&\text{if $d-k= 0$, i.e., $k=\height(\p)$}
\end{cases}
$$
The second step follows from the fact that $E$ is $\m$-torsion and injective.
\end{proof}

\begin{proposition}
\label{20141009.1}
Let $R$ be Cohen-Macaulay with a dualizing module $D$, and let $\p \in \spec(R)$. Then for each injective $R$-module $E$, 
there is a set $\mu_\p$ such that
$$\tor_k^R(\hom_R(D,E_R(R/\p)),E) 
\cong\begin{cases}0&\text{if  $k\neq \height(\p)$}\\ E_R(R/\p)^{(\mu_p)}&\text{if $k= \height(\p)$.}\end{cases}
$$ 
\end{proposition}

\begin{proof}
Recall that we may express $E$ as a direct sum $E \cong \bigoplus_\q E_R(R/\q)^{(\mu_\q)}$ indexed over  $\spec(R)$; since $\tor_k^R(-,\bigoplus_\q E_R(R/\q)^{(\mu_\q)}) \cong \bigoplus_{\q}\tor_k^R(-,E_R(R/\q))^{(\mu_\q)}$, 
the desired result follows from Proposition~\ref{prop141231a}.
\end{proof}

\begin{corollary}
\label{20141010.1}
Let $R$ be Cohen-Macaulay with a dualizing module $D$. If $E$ and $E'$ are injective $R$-modules, then $\tor_i^R(\hom_R(D,E'),E)$ is injective for all $i \geq 0$.
\end{corollary}

\begin{proof}
Since $D$ is finitely generated over $R$, it suffices to show that the $R$-module
$\tor_i^R(\hom_R(D,E_R(R/\p)),E_R(R/\q))$ is injective for all prime ideals $\p,\q$ and for all $i \geq 0$. 
The result in this case follows from Proposition~\ref{20141009.1}.
\end{proof}

The next result generalizes the vanishing part of Proposition~\ref{20141009.1}.

\begin{proposition}
\label{20141003.2}
Let $R$ be Cohen-Macaulay with a dualizing module $D$, and let $\p \in \spec(R)$. Then for any Gorenstein injective $R$-module $G$ we have 
$$\tor_k^R(\hom_R(D,E_R(R/\p)),G) = 0$$ 
for all $k \neq \height(\p)$.
\end{proposition}

\begin{proof}
We have by Lemma~\ref{20141030.1} that $\fd_R(\hom_R(D,E_R(R/\p))) \leq \height(\p)$, so for $k > \height(\p)$ we have the vanishing $\tor_k^R(\hom_R(D,E_R(R/\p)),G) = 0$. We show the vanishing when $k < \height(\p)$ by induction on $k$.

For the base case $k = 0$, we have $$\tor_0^R(\hom_R(D,E_R(R/\p)),G) = \hom_R(D,E_R(R/\p)) \otimes_R G = 0$$ by Lemma~\ref{20141002.3} provided that $\height(\p) \geq 1$.

For the induction step, observe that there is an exact sequence $$0 \to H \to E \to G \to 0$$ for some injective $R$-module $E$ and Gorenstein injective $R$-module $H$. The long exact sequence in the functor $\tor^R(\hom_R(D,E_R(R/\p)),-)$ gives rise to the exact sequence 
\begin{multline*}
\tor_k^R(\hom_R(D,E_R(R/\p)),E) \to \tor_k^R(\hom_R(D,E_R(R/\p)),G) \\ \to \tor_{k-1}^R(\hom_R(D,E_R(R/\p)),H)
\end{multline*}
 where, by our induction hypothesis, we have $\tor_{k-1}^R(\hom_R(D,E_R(R/\p)),H) = 0$ since $H$ is Gorenstein injective. We have $\tor_k^R(\hom_R(D,E_R(R/\p)),E) = 0$ 
 by Proposition~\ref{20141009.1}, so $\tor_k^R(\hom_R(D,E_R(R/\p)),G) = 0$ as desired.
\end{proof}

\begin{corollary}
\label{20141010.2}
Let $R$ be Cohen-Macaulay with a dualizing module $D$. 
Given an exact sequence $0 \to G' \to G \to G'' \to 0$ of Gorenstein injective $R$-modules and an injective $R$-module $E$, the sequence $$0 \to \tor_k^R(\hom(D,E),G') \to \tor_k^R(\hom(D,E),G) \to \tor_k^R(\hom(D,E),G'') \to 0$$ is exact for all $k \geq 0$.
\end{corollary}

\begin{proof}
It suffices to show the result when $E := E_R(R/\p)$ for some fixed prime $\p$ such that $\height{\p} = k$ by Proposition~\ref{20141003.2}. Consider
part of the long exact sequence in $\tor^R(\hom_R(D,E),-)$:
\begin{multline*}
\tor_{k+1}^R(\hom(D,E),G'') \to \tor_k^R(\hom(D,E),G') \to \tor_k^R(\hom(D,E),G) \\ \to \tor_k^R(\hom(D,E),G'') \to 
\tor_{k-1}^R(\hom(D,E),G').
\end{multline*}
By Proposition~\ref{20141003.2}, this sequence is short exact.
\end{proof}

The next result augments another part of Proposition~\ref{20141009.1}.

\begin{proposition}
\label{20141010.3}
Let $R$ be Cohen-Macaulay with a dualizing module $D$. If $G$ is a Gorenstein injective $R$-module and $E$ is any injective $R$-module, then 
the module $\tor_k^R(\hom_R(D,E),G)$ is Gorenstein injective for all $k \geq 0$.
\end{proposition}

\begin{proof}
Since $R$ has a dualizing module, the class of Gorenstein injective $R$-modules is closed under direct sums by~\cite[Theorem 6.10]{3danes}. Thus, it suffices to show the case when $E = E_R(R/\p)$ for some prime ideal $\p$. Since $G$ is Gorenstein injective, there is an exact sequence $\cdots \to E_1 \to E_0 \to G \to 0$ such that each $E_i$ is injective (and hence Gorenstein injective) and the kernel $K_i$ of each map is Gorenstein injective. If we split this sequence into short exact sequences $0 \to K_0 \to E_0 \to G \to 0$ and $0 \to K_i \to E_i \to K_{i-1} \to 0$, we combine the corresponding short exact sequences from Corollary~\ref{20141010.2} to obtain the following exact sequence:
$$\cdots \to \tor_k(\hom(D,E),E_1) \to \tor_k(\hom(D,E),E_0) \to \tor_k(\hom(D,E),G) \to 0$$ 
Since each $\tor_k^R(\hom_R(D,E),E_i)$ is injective by Corollary~\ref{20141010.1}, we claim that $\tor_k^R(\hom_R(D,E),G)$ is Gorenstein injective.
By Lemma~\ref{lem141223a}\eqref{lem141223a2}, it suffices to show that that each module
$\tor_k^R(\hom_R(D,E),E_i)$ is Gorenstein injective and that we have $\gid_R(\tor_k^R(\hom_R(D,E),G)) < \infty$.

To  this end, let $$F := \quad 0 \to F_n \to \cdots \to F_0 \to 0$$ be a flat resolution of $\hom_R(D,E)$, and consider the sequence $$F \otimes_R G = \quad 0 \to F_n \otimes_R G \xrightarrow{\partial_n} \cdots \xrightarrow{\partial_1} F_0 \otimes_R G \to 0.$$ Note that $H_i(F \otimes_R G) = \tor_i^R(\hom_R(D,E),G) = 0$ for $i \neq k$. Since $G \in \bd$ by Fact~\ref{fact141223c}, and each $F_i$ is flat, it can be shown that each $F_i \otimes_R G \in \bd$. Repeated application of the two-of-three property gives that the only possibly non-vanishing homology module of $F \otimes_R G$ is in $\bd$; in other words, $\tor_k^R(\hom_R(D,E),G) \in \bd$. Fact~\ref{fact141223c} gives that $\gid_R(\tor_k^R(\hom_R(D,E),G)) < \infty$. 
Lastly, note that $\tor_k^R(\hom(D,E),E_i)$ is Gorenstein injective by Example~\ref{ex141223a} and Corollary~\ref{20141010.1}.
\end{proof}

We conclude this section with a technical lemma.

\begin{lemma}\label{lem141224d}
Let $M$ and $N$ be  $R$-modules, and let $\p\in\spec(R)$ such that $M$ satisfies $t(\p)$.
Let $X_1,\ldots,X_t$ be pair-wise disjoint subsets of $\spec(R)$, and assume that $N$ has a filtration
$0=N_{t+1}\subseteq N_t\subseteq\cdots\subseteq N_1=N$ such that each quotient $N_j/N_{j+1}$
decomposes as $N_j/N_{j+1}\cong \oplus_{\q\in X_j}N_{j,\q}$ where $N_{j,\q}$ satisfies $\tp[q]$.
\begin{enumerate}[\rm(a)]
\item \label{lem141224d1}
If $\p\notin\bigcup_{j=1}^tX_j$, then $\Tor iMN=0$ for all $i$.
\item \label{lem141224d2}
If $\p\in X_j$ for some $j$, then $\Tor iMN\cong\Tor iM{N_{j,\p}}$ for all $i$.
\item \label{lem141224d3}
One has $\supp_R(N)=\{\q\in\bigcup_{j=1}^tX_j\mid N_{j,\q}\neq 0\}$.
\end{enumerate}
\end{lemma}

\begin{proof}
\eqref{lem141224d1}
We argue by induction on $t$.

Base case: $t=1$. Our assumption implies that $N=N_1\cong N_1/N_2\cong\oplus_{\q\in X_1}N_{1,\q}$.
Thus, by Lemma~\ref{lem141224c} we have
\begin{align*}
\Tor iMN
&\cong\oplus_{\q\in X_1}\Tor iM{N_{1,\q}}=0.
\end{align*}

Induction step: $t>1$. 
The module $N_2$ satisfies the induction hypothesis, so we have $\Tor iM{N_2}=0$ for all $i$.
The module $N/N_2$ satisfies the base case, so we have $\Tor iM{N/N_2}=0$ for all $i$.
From the short exact sequence $0\to N_2\to N\to N/N_2\to 0$, the long exact sequence in $\Tor {}M-$ shows that $\Tor iMN=0$ for all $i$.
(Note that this does not use the fact that the sets $X_j$ are pair-wise disjoint.)

\eqref{lem141224d2}
Consider the short exact sequence
$0\to N_{j+1}\to N_j\to N_j/N_{j+1}\to 0$.
Part~\eqref{lem141224d1} shows that $\Tor iM{N_{j+1}}=0$ for all $i$,
so the long exact sequence in $\Tor {}M-$ provides isomorphisms
$$\Tor iM{N_j}\cong\Tor iM{N_j/N_{j+1}}$$ 
for all $i$.
Similarly, using the short exact sequence
$0\to N_j\to N\to N/N_j\to 0$,
we conclude that 
$$\Tor iM{N}\cong\Tor iM{N_j}\cong\Tor iM{N_j/N_{j+1}}$$ 
for all $i$.
Now, from the direct sum decomposition 
$N_j/N_{j+1}\cong \oplus_{\q\in X_j}N_{j,\q}$
we have
$$\Tor iM{N}\cong\Tor iM{N_j/N_{j+1}}\cong\oplus_{\q\in X_j}\Tor iM{N_{j,\q}}\cong\Tor iM{N_{j,\p}}$$ 
where the last isomorphism follows from Lemma~\ref{lem141224c}.

\eqref{lem141224d3}
For one containment, let $\p\in\{\q\in\bigcup_{j=1}^tX_j\mid N_{j,\q}\neq 0\}$.
From part~\eqref{lem141224d2}, we have the next isomorphism
\begin{equation}\label{eq141224a}
\Tor i{\kp}N\cong\Tor i{\kp}{N_{j,\p}}
\end{equation}
for all $i$. From Lemma~\ref{lem141224a}, we have $\p\in\supp_R(N_{j,\p})$;
hence, the second Tor-module in~\eqref{eq141224a} is non-zero for some $i$.
This implies that $\p\in\supp_R(N)$.

For the reverse containment, let $\p\in\spec(R)\smallsetminus\{\q\in\bigcup_{j=1}^tX_j\mid N_{j,\q}\neq 0\}$.
If we have $\p\notin\bigcup_{j=1}^tX_j$, then part~\eqref{lem141224d1} shows that $\Tor i{\kp}N=0$ for all $i$, so
$\p\notin\supp_R(N)$.
Thus, we assume that $\p\in\bigcup_{j=1}^tX_j$. Then we must have $N_{j,\p}=0$,
so the logic of the preceding paragraph implies that 
$\Tor i{\kp}N\cong\Tor i{\kp}{N_{j,\p}}=0$ for all $i$; that is, $\p\notin\supp_R(N)$.
\end{proof}

\section{Filtrations for Gorenstein Injective Modules}\label{sec141229a}

We begin this section with some convenient notation. 

\begin{notn}\label{notn141224a}
Let $R$ be Cohen-Macaulay with a dualizing module $D$, and 
let $G$ be a Gorenstein injective $R$-module.
For each prime $\p\in\spec(R)$,  set 
$$G_{(\p)}:=\Tor k{\hom_R(D,E_R(R/\p))}{G}\cong \Gamma_\p(G_\p)$$
where 
$k=\height(\p)$;
see Proposition~\ref{prop141231b}.
Note that $G_{(\p)}$ is Gorenstein injective by Proposition~\ref{20141010.3}.
\end{notn}

The next result is  Theorem~\ref{20141114.1} from the introduction, which generalizes~\cite[Theorem 3.1]{enochs}.

\begin{theorem}
\label{20141010.4}
Let $R$ be a $d$-dimensional Cohen-Macaulay ring with a dualizing module $D$. If $G$ is a Gorenstein injective $R$-module, then $G$ has a filtration $$0 = G_{d+1} \subset G_d \subset \cdots \subset G_1 \subset G_0 = G$$ such that each submodule $G_k$ and each quotient 
$G_k/G_{k+1}\cong\oplus_{\height(\p)=k}G_{(\p)}$
is Gorenstein injective and 
each module 
$G_{(\p)}$
is Gorenstein injective and 
satisfies $\tp$.
Moreover, this filtration and the direct sum compositions of the factors are unique and functorial.
\end{theorem}

\begin{proof}
We use a  spectral sequence argument like the proof of~\cite[Theorem 3.1]{enochs}. Let $${^+}J := \quad 0 \to D \to J^0 \to J^1 \to \cdots \to J^d \to 0$$ be the minimal injective resolution of $D$, where $J^i := \bigoplus_{\height(\p)=i} E_R(R/\p)$. We have the exact sequence
$$\hom(D,{^+}J) = \quad 0 \to R \to \hom(D,J^0) \to \hom(D,J^1) \to \cdots \to \hom(D,J^d) \to 0$$ since $\hom_R(D,D) \cong R$ and $\ext_R^i(D,D) = 0$ for $i \geq 1$. Let $P^+ := \cdots \to P_1 \to P_0 \to G \to 0$ be a projective resolution of $G$, and form the following double complex.
$$\xymatrix@C=5mm{
& 0 & 0 & & 0 \\
0 \ar[r] & \hom(D,J^0) \otimes P_0 \ar[r]\ar[u] & \hom(D,J^1) \otimes P_0 \ar[r]\ar[u] & \cdots \ar[r] & \hom(D,J^d) \otimes P_0 \ar[r]\ar[u] & 0 \\
0 \ar[r] & \hom(D,J^0) \otimes P_1 \ar[r]\ar[u] & \hom(D,J^1) \otimes P_1 \ar[r]\ar[u] & \cdots \ar[r] & \hom(D,J^d) \otimes P_1 \ar[r]\ar[u] & 0 \\
& \vdots \ar[u] & \vdots \ar[u] & & \vdots \ar[u]
}$$
We index this complex such that $\hom_R(D,J^i) \otimes_R P_j$ has index $(-i,j)$. When we first compute horizontal homology,
to find the $E^1$ page of the spectral sequence, we obtain modules of the form 
$$H_i(\hom_R(D,J) \otimes_R P_j) \cong H_i(\hom_R(D,J)) \otimes_R P_j$$ 
since each $P_j$ is flat. We hence have 
$$H_i(\hom_R(D,J) \otimes_R P_j) \cong \begin{cases}
R \otimes_R P_j \cong P_j & \text{if $i = 0$} \\ 0 & \text{if $i \neq 0$.} \end{cases}$$
When we compute vertical homology on the $E^1$-page, to find the $E^2$-page of the spectral sequence, we obtain $G$ in the $(0,0)$ index and $0$ elsewhere.

When we instead first compute vertical homology of our double complex, we obtain modules $H_j(\hom_R(D,J^i) \otimes_R P) = \tor_j^R(\hom_R(D,J^i),G)$. Since $J^i = \bigoplus_{\height(\p)=i} E_R(R/\p)$, Proposition~\ref{20141009.1} gives that $\tor_j^R(\hom_R(D,J^i),G) = 0$ unless $i = j$. 
Thus, the $E^1$ page of this spectral sequence is concentrated on a diagonal.
When we compute horizontal homology here (to find the $E^2$-page of the spectral sequence), we obtain modules $\tor_k^R(\hom_R(D,J^k),G)$ in the $(-k,k)$ index and $0$ elsewhere.

This means that $G$ has a filtration $0 = G_{d+1} \subset G_d \subset \cdots \subset G_1 \subset G_0 = G$ such that $G_k/G_{k+1} \cong \tor_k^R(\hom_R(D,J^k),G)$ for $0 \leq k \leq d$. Lemma~\ref{20141010.3} gives that each of these quotients is Gorenstein injective.
In particular, the module
$G_d\cong G_d/G_{d+1}$ is Gorenstein injective.
Since the same is true of $G_{d-1}/G_d$,
the short exact sequence $0\to G_d\to G_{d-1}\to G_{d-1}/G_d\to 0$ shows that $G_{d-1}$ is Gorenstein injective.
Working our way up the filtration in a similar manner, we conclude that each module $G_k$ is Gorenstein injective as well.

Since $J^k = \bigoplus_{\height(\p)=k} E_R(R/\p)$ for each $k$, there is a natural isomorphism $$\tor_k^R(\hom_R(D,J^k),G) \cong \bigoplus_{\height(\p)=k} \tor_k^R(\hom_R(D,E_R(R/\p)),G).$$ Since we have shown that $\hom_R(D,E_R(R/\p))$ has property $t(\p)$, 
it follows that $\tor_k^R(\hom_R(D,E_R(R/\p)),G)$ has property $\tp$ as well.
Again, Lemma~\ref{20141010.3} implies that each module  
$G_{(\p)}=\tor_k^R(\hom_R(D,E_R(R/\p)),G)$
is Gorenstein injective.

The uniqueness and functoriality of the filtrations and decompositions are established exactly as in the proof of~\cite[Theorem 3.1]{enochs}.
\end{proof}

\begin{proposition}\label{prop141224a}
Let $R$ be Cohen-Macaulay with a dualizing module $D$, and let $G$ be a Gorenstein injective $R$-module.
Then $\supp_R(G)=\{\p\in\spec(R)\mid G_{(\p)}\neq 0\}$.
\end{proposition}

\begin{proof}
This follows from Lemma~\ref{lem141224d}\eqref{lem141224d3} since  $G_{(\p)}$ satisfies $\tp$.
\end{proof}

\begin{rmk}\label{rmk141231b}
If $I$ is an injective module, 
then the above result partially recovers Matlis' decomposition $I\cong\oplus_\p E_R(R/\p)^{(\mu_\p)}$.
Indeed, since $I$ is injective, each summand $I_{(\p)}\cong \Gamma_\p(I_{\p})$ is injective and satisfies $\tp$. 
Moreover, given such a decomposition $I\cong\oplus_\q E_R(R/\q)^{(\mu_\q)}$,by Fact~\ref{fact141231a}  we must have
$$I_{(\p)}\cong\bigoplus_\q \Gamma_\p(E_R(R/\q)_{\p})^{(\mu_\q)}\cong E_R(R/\p)^{(\mu_\p)}.$$
\end{rmk}

\section{Examples}\label{sec150430a}

Given the parallels between our results and those from~\cite{enochs}, it would be natural to expect that the class of Gorenstein injective modules is closed under tensor products in our setting.
The following example shows that this is not the case for non-Gorenstein rings in general.
Moreover, it shows that the ``generically Gorenstein'' assumption in Theorem~\ref{20141114.2} is necessary.

\begin{example}\label{ex141218a}
Let $k$ be a field and set $R:=k[X_1,\ldots,X_d]/(X_1,\ldots,X_d)^2$ with $d\geq 2$.
Then $R$ is a local ring with maximal ideal $\m:=(X_1,\ldots,X_d)R$ and residue field $k$
such that $\m^2=0$. In particular, $R$ is ``connected''; that is, it is not a proper direct product of rings.
Since $R$ is not Gorenstein, we know from~\cite[Theorem 1.1]{chen} that each finitely generated
Gorenstein projective $R$-module is projective; hence, it is free, since $R$ is local. Also, as $R$ is an artinian local ring,
it is Cohen-Macaulay with dualizing module $E:=E_R(k)$. 

Given the specific form of $R$, we know that $E$ is finitely generated over $R$ by elements
$x_1^*,\ldots,x_d^*$ subject to the relations $x_ix_j^*=0$ when $i\neq j$ and such that $x_ix_i^*=x_jx_j^*$ for all $i,j$.
From this, it follows that $\m (E\otimes_RE)=0$.
Furthermore, Nakayama's Lemma implies that $E\otimes_RE$ is minimally generated by the $d^2$ elements of the form
$x_i^*\otimes x_j^*$. Thus, we have $E\otimes_RE\cong k^{d^2}$.

Since $E$ is injective over $R$, it is Gorenstein injective. On the other hand, the tensor product $E\otimes_RE\cong k^{d^2}$ is not Gorenstein injective
(in fact, it has infinite Gorenstein injective dimension) as follows. Suppose that $\gid_R(k^{d^2})<\infty$. 
Then $\gid_R(k^{d^2})\leq\dim(R)=0$ by Lemma~\ref{lem141223a}\eqref{lem141223a1}, so $k^{d^2}$ is Gorenstein injective.
It follows that $k$ is Gorenstein injective as well.
We conclude from~\cite[5.1]{3danes} that $k\cong\hom_R(k,E)$ has finite Gorenstein flat dimension,
so it is Gorenstein projective by~\cite[1.4, 3.1, 3.8]{3danes}.
On the other hand, $k$ is not free over $R$, so this contradicts the fact that 
finitely generated
Gorenstein projective $R$-modules must be free.
\end{example}

The next example gives a negative answer to the following question implicit in~\cite[Remark 4.2]{enochs}:
if $R$ is Gorenstein and $G,H$ are Gorenstein injective, must $\Tor iGH$ be Gorenstein injective for all $i$?

\begin{example}\label{ex141231a}
Let $k$ be a field and set $R:=k[\![X,Y]\!]/(XY)$ with maximal ideal $\m:=(X,Y)R$. Also, set $E:=E_R(k)$ and $(-)^\vee:=\hom_R(-,E)$.
The modules $R/XR$ and $R/YR$ are finitely generated and Gorenstein projective.
It follows that the duals $G:=(R/XR)^\vee$ and $H:=(R/YR)^\vee$ are Gorenstein injective;
moreover, they are non-zero since $E$ is faithfully injective.
Note that $XG=0=YH$.
It follows that $\m\Tor iGH=0$ for all $i$.
Also, since $R/XR$ and $R/YR$ are finitely generated, and $E$ satisfies $\tp[m]$, we conclude that $G$ and $H$ satisfy $\tp[m]$ as well.

Lemma~\ref{lem141224a} implies that $\supp_R(G)=\{\m\}=\supp_R(H)$.
Thus, we conclude from Fact~\ref{fact141224a} that $\Tor iGH\neq 0$ for some $i$. 
Lemma~\ref{20141002.2} implies that $G\otimes_RH=0$, so we must have $i\geq 1$ here.

Claim: We have
\begin{equation}\label{eq141231a}
\Tor iGH\cong
\begin{cases}
k & \text{if $i=2,4,6,\ldots$} \\
0 & \text{otherwise.}
\end{cases}
\end{equation}
Before proving the claim, we show how it deals with~\cite[Remark 4.2]{enochs}.
Suppose that $\Tor 2GH\cong k$ were Gorenstein injective.
We know that finitely generated modules of finite Gorenstein injective dimension must have
Gorenstein injective dimension equal to $\operatorname{depth}(R)=1$, by~\cite[Corollary 2.5]{MR2495252}, a contradiction.

Now we prove the claim.

Step 1:
We have  exact sequences
\begin{gather}
\label{eq141231b}
0\to k\to H\xrightarrow XH\to 0\\
\label{eq141231b'}
0\to k\to G\xrightarrow YG\to 0.
\end{gather}
Indeed, as a $k$-vector space, $E$ has a basis
$\{1,X^{-1},Y^{-1},X^{-2},Y^{-2},\ldots\}$.
Write $1=X^0=Y^0$.
With this basis, the $R$-module structure on $E$ is given by the following, where $p,q,i,j\geq 0$:
\begin{align*}
X^{p}X^{-i}
&=\begin{cases}
X^{p-i}&\text{if $p\leq i$}\\
0&\text{if $p>i$}
\end{cases}
&Y^{q}Y^{-j}
&=\begin{cases}
X^{q-j}&\text{if $q\leq j$}\\
0&\text{if $q>j$}
\end{cases}
\\
X^{p}Y^{-j}
&=\begin{cases}
1&\text{if $p=0=j$}\\
0&\text{otherwise}
\end{cases}
&Y^{q}X^{-i}
&=\begin{cases}
1&\text{if $q=0=i$}\\
0&\text{otherwise.}
\end{cases}
\end{align*}
Using the isomorphism
$$H:=\hom_R(R/YR,E)\cong\{\alpha\in E\mid Y\alpha=0\}$$
it is straightforward to show that
$H$ has a $k$-basis $\{1,X^{-1},X^{-2},\ldots\}$,
with the $R$-module structure on $H$ determined by the relations above
(since we are identifying $H$ with a submodule of $E$).
From this, it follows that the map $H\xrightarrow XH$ is onto with kernel given by $k\cdot 1\cong k$.
This establishes the exact sequence~\eqref{eq141231b}; the other one is established by symmetry.
This concludes Step 1. 

Step 2:
We have
\begin{equation*}
\Tor iGk\cong
\begin{cases}
k & \text{if $i\geq 1$} \\
0 & \text{if $i=0$.}
\end{cases}
\end{equation*}
Indeed, consider the following truncated $R$-free resolution of $k$:
\begin{equation*}
F:=\cdots
\xrightarrow{\left(\begin{smallmatrix}Y & 0 \\ 0 & X\end{smallmatrix}\right)}R^2
\xrightarrow{\left(\begin{smallmatrix}X & 0 \\ 0 & Y\end{smallmatrix}\right)}R^2
\xrightarrow{\left(\begin{smallmatrix}Y & 0 \\ 0 & X\end{smallmatrix}\right)}R^2
\xrightarrow{\left(\begin{smallmatrix}X & Y\end{smallmatrix}\right)}R
\to 0.
\end{equation*}
Since $XG=0$, tensoring with $G$ yields the following complex:
\begin{equation*}
G\otimes_RF\cong\cdots
\xrightarrow{\left(\begin{smallmatrix}Y & 0 \\ 0 & 0\end{smallmatrix}\right)}G^2
\xrightarrow{\left(\begin{smallmatrix}0 & 0 \\ 0 & Y\end{smallmatrix}\right)}G^2
\xrightarrow{\left(\begin{smallmatrix}Y & 0 \\ 0 & 0\end{smallmatrix}\right)}G^2
\xrightarrow{\left(\begin{smallmatrix}0 & Y\end{smallmatrix}\right)}G
\to 0.
\end{equation*}
Using the exact sequence~\eqref{eq141231b'}, one readily verifies the desired conclusions about $\Tor iGk$ from the above description of $G\otimes_RF$.
This concludes Step 2.

Step 3: We verify the claim (hence, concluding the example) by verifying the isomorphisms in~\eqref{eq141231a}.
For the case $i=0$, since $G$ and $H$ satisfy $\tp[m]$ we have $G\otimes_RH=0$ from Lemma~\ref{20141002.2}.
Next, consider the long exact sequence in $\Tor{}{G}{-}$ associated to
the short exact sequence~\eqref{eq141231b}. 
Given the established vanishing  $G\otimes_Rk=0=G\otimes_R H$, this long exact sequence begins as follows:
$$\Tor 1GH\xrightarrow[=0]{X}\Tor 1GH\to 0.$$
The map here is 0 as $XG=0$.
The exactness of this sequence implies $\Tor 1GH=0$.
Further in the long exact sequence, for $i\geq 1$, we have the following:
$$\Tor{i+1}GH\xrightarrow 0\Tor{i+1}GH\to\Tor iGk\to \Tor iGH\xrightarrow 0\Tor iGH.$$
In other words, we have the following short exact sequence
\begin{equation}
\label{eq141231e}
0\to\Tor{i+1}GH\to k\to \Tor iGH\to 0.
\end{equation}
As $\Tor 1GH=0$, the sequence~\eqref{eq141231e} for $i=1$ implies that $\Tor 2GH\cong k$. 
From this, the sequence for $i=2$ implies that $\Tor 3GH=0$.
The remaining $\Tor iGH$ follow similarly, say, by induction on $i$.
\end{example}

The next example shows that the filtration from Theorem~\ref{20141010.4} need not yield a direct sum decomposition of $G$;
that is, we can have $G\not\cong\oplus_k G_k/G_{k+1}\cong\oplus_\p G_{(\p)}$.
This is in stark contrast with Matlis' decomposition result for injective modules.

\begin{example}\label{ex150430a}
Let $k$ be a field, and set $R:=k[\![X,Y]\!]/(X^2)$ with maximal  ideal $\m:=(X,Y)R$ and $E:=E_R(k)$.
Set $\q:=(X)R$ and $\ol R:=R/\q$ with $\ol\m:=\m\ol R$, and note that $\spec(R)=\{\q,\m\}$.
Consider the natural inclusion $\ol R\subseteq E_R(\ol R)$, and set $G:=E_R(\ol R)/\ol R$.

We claim that $G$ is an indecomposable Gorenstein injective $R$-module with $G_{(\q)},G_{(\m)}\neq 0$.
Before proving the claim, we note that it immediately implies 
that we have $G\not\cong\oplus_\p G_{(\p)}\cong\oplus_k G_k/G_{k+1}$.
Now we prove the claim. 

Step 1: $G$ is Gorenstein injective. Since $R$ is a Gorenstein ring of dimension 1 and $E_R(\ol R)$ is injective over $R$,
this follows from~\cite[Theorem 4.1]{enochs:gipm} or Lemma~\ref{lem141223a}\eqref{lem141223a2}.

Step 2: $G_{(\q)}\cong\kappa(\q)\neq 0$. By construction, we have $\ol R_{\q}\cong \kappa(\q)$.
Also, we have 
$$E_R(\ol R)=E_R(R/\q)\cong E_{R_{\q}}(\kappa(\q))\cong R_{\q}$$
where the first isomorphism is standard, and the second isomorphism is from the fact that $R_{\q}$ is artinian, local, and Gorenstein.
The natural exact sequence
\begin{equation}\label{eq150430a}
0\to\ol R\to E_R(\ol R)\to G\to 0
\end{equation}
then localizes to an exact sequence of the  form
$$0\to\kappa(\q)\to R_{\q}\to G_{\q}\to 0.$$
Since $R_{\q}$ is artinian, local, and Gorenstein, this sequence must have the form
$$0\to\operatorname{Soc}(R_{\q})\to R_\q\to G_{(\q)}\to 0.$$
The fact that the maximal ideal of $R_\q$ is square-zero then implies that
$G_{(\q)}\cong R_\q/\q R_{\q}=\kappa(\q)$, as desired.

Step 3: $G_{(\m)}\cong E_{\ol R}(k)\neq 0$.
Note that we have $G_{(\m)}\cong\Gamma_{\m}(G_\m)=\Gamma_\m(G)$.
Since $E_R(\ol R)$ satisfies $\tp[q]$ with $\q\subsetneq\m$, we have $\Gamma_{\m}(E_R(\ol R)_\m)=\Gamma_\m(E_R(\ol R))=0$.
Also, we have $H^i_{\m}(G_\m)=0=H^i_{\m}(E_R(\ol R)_\m)$ for all $i\neq 0$ by Propositions~\ref{prop141231b} and~\ref{20141003.2}.
Thus, the long exact sequence in $H^i_\m(-)$ associated to~\eqref{eq150430a} explains the second isomorphism in the next sequence:
\begin{align*}
G_{(\m)}
&\cong \Gamma_\m(G)
\cong H^1_\m(\ol R)
\cong H^1_{\ol\m}(\ol R)
\cong E_{\ol R}(k).
\end{align*}
The third isomorphism follows from the standard independence-of-base-ring theorem for local cohomology,
and the fourth isomorphism is from the fact that $\ol R$ is a 1-dimensional local Gorenstein ring.
This completes Step 3.

Step 4: $G\not\cong G_{(\m)}\oplus G_{(\q)}$.
Note that Steps 3 and 4 imply, in particular, that $XG_{(\m)}=0=X G_{(\q)}$.
Thus, it suffices to show that $XG\neq 0$.
Suppose by way of contradiction that $XG=0$.
By definition of $G$, this implies that $XE_R(\ol R)\subseteq\ol R$.
However, we have
$XE_R(\ol R)=XE_{R_\q}(\kappa(\q))\cong\kappa(\q)$
since $R_\q$ is an artinian, local, Gorenstein ring with square-zero maximal ideal $X R_\q\neq 0$.
Since $\ol R$ is a 1-dimensional domain, it cannot contain a copy of its quotient field $\kappa(\q)$,
contradicting the containment $XE_R(\ol R)\subseteq\ol R$.
This completes Step 4.

Step 5: $G$ is indecomposable. Suppose that $G\cong G'\oplus G''$ with $0\neq G',G''\subseteq G$.
Since $G$ is Gorenstein injective over $R$, so are $G'$ and $G''$.
It follows that $G'_{(\m)}= \Gamma_\m(G')=G'\cap\Gamma_\m(G)=G'\cap G_{(\m)}$,
and similarly $G''_{(\m)}=G''\cap G_{(\m)}$. Since the sum $G'+G''$ is direct, we have $G'\cap G''=0$;
hence, the previous sentence implies that $G'_{(\m)}\cap G_{(\m)}''=0$, and furthermore that
$G_{(\m)}=G'_{(\m)}+G''_{(\m)}$. Thus, we have $G_{(\m)}=G'_{(\m)}\oplus G''_{(\m)}$.
Step 3 shows that $G_{(\m)}$ is indecomposable, so (by symmetry) we must have
$G_{(\m)}=G'_{(\m)}$ and $G''_{(\m)}=0$.
Since we have assumed that $G''\neq 0$, our filtration implies that $0\neq G''\cong G''_{(\q)}$.

The functoriality of our filtration yields the following commutative diagram with exact rows and columns:
$$\xymatrix{
& 0\ar[d]
& 0\ar[d]
& 0\ar[d]
\\
0\ar[r]
&G''_{(\m)}\ar[r]\ar[d]
&G''\ar[d]\ar[r]
&G''_{(\q)}\ar[r]\ar[d]
&0
\\
0\ar[r]
&G_{(\m)}\ar[r]\ar[d]
&G\ar[d]\ar[r]
&G_{(\q)}\ar[r]\ar[d]
&0
\\
0\ar[r]
&G'_{(\m)}\ar[r]\ar[d]
&G'\ar[d]\ar[r]
&G'_{(\q)}\ar[r]\ar[d]
&0
\\
&
0&0&0.}$$
The facts we have established above imply that this has the following form:
$$\xymatrix{
& 0\ar[d]
& 0\ar[d]
& 0\ar[d]
\\
0\ar[r]
&0\ar[r]\ar[d]
&G''\ar[d]\ar[r]
&G''_{(\q)}\ar[r]\ar[d]^\alpha
&0
\\
0\ar[r]
&G_{(\m)}\ar[r]\ar[d]
&G\ar[d]\ar[r]
&\kappa(\q)\ar[r]\ar[d]
&0
\\
0\ar[r]
&G'_{(\m)}\ar[r]\ar[d]
&G'\ar[d]\ar[r]
&G'_{(\q)}\ar[r]\ar[d]
&0
\\
&
0&0&0.}$$
Since $G''_{(\q)}$ is a non-zero $R_{\q}$-module, and $\kappa(\q)$ is a simple $R_\q$-module, the map $\alpha$ must be surjective,
so we have $G'_{(\q)}=0$. Hence, our diagram has the next form:
$$\xymatrix{
& 0\ar[d]
& 0\ar[d]
& 0\ar[d]
\\
0\ar[r]
&0\ar[r]\ar[d]
&G''\ar[d]\ar[r]
&G''_{(\q)}\ar[r]\ar[d]^\alpha
&0
\\
0\ar[r]
&G_{(\m)}\ar[r]\ar[d]
&G\ar[d]\ar[r]
&\kappa(\q)\ar[r]\ar[d]
&0
\\
0\ar[r]
&G'_{(\m)}\ar[r]\ar[d]
&G'\ar[d]\ar[r]
&0\ar[r]\ar[d]
&0
\\
&
0&0&0.}$$
Thus, we conclude  that
$G'=G'_{(\m)}=G_{(\m)}$ and $G''=G'_{(\q)}=G_{(\q)}$.
It follows that the given direct sum decomposition then has the form
$G\cong G_{(\m)}\oplus G_{(\q)}$, contradicting Step 4.
\end{example}

We end this section with a natural question.

\begin{question}
Given a Gorenstein injective $R$-module $G$, does there exist an isomorphism $G\cong\oplus_\lambda G_\lambda$
where each $G_\lambda$ is indecomposable? (Note that, given such a decomposition, each $G_\lambda$ is automatically Gorenstein injective.)
\end{question}

\section{$C$-Gorenstein Injective Results}
\label{sec141229b}

In this section, we prove a $C$-Gorenstein injective version of Theorem~\ref{20141114.1} as well as Theorem~\ref{20141114.2}.

\begin{assumption}
We assume for this section that $C$ is a semidualizing $R$-module.
\end{assumption}

\begin{theorem}
\label{cor141225a}
Let $R$ be a $d$-dimensional Cohen-Macaulay ring with a dualizing module $D$. If $G$ is a $C$-Gorenstein injective  $R$-module, then $G$ 
has a filtration $$0 = G_{d+1} \subset G_d \subset \cdots \subset G_1 \subset G_0 = G$$ such that each submodule $G_k$ and each quotient 
$G_k/G_{k+1}\cong\oplus_{\height(\p)=k}G_{(\p)}$
is $C$-Gorenstein injective and 
each module 
$G_{(\p)}\cong \Gamma_\p(G_\p)$
is $C$-Gorenstein injective and 
satisfies $\tp$.
Moreover, this filtration and the direct sum compositions of the factors are unique and functorial.
\end{theorem}

\begin{proof}
Set $S=R\ltimes C$.
Since $R$ is Cohen-Macaulay, the $R$-module $C$ is locally maximal Cohen-Macaulay,
so $S$ is Cohen-Macaulay. 
Also, $R$ is a homomorphic image of a finite-dimensional Gorenstein ring;
since $S$ is a module-finite $R$-algebra, $S$ is also a homomorphic image of a finite-dimensional Gorenstein ring.
Thus, $S$ has a dualizing module. 
(It is straightforward to show that $\hom_R(C,D)\oplus D$ is dualizing for $S$, but we do not need this here; see~\cite{jorgensen:prnsm}.)

The $R$-module $G$ is Gorenstein injective over $S$ by Fact~\ref{fact141223b}.
Thus, Theorem~\ref{20141010.4} implies that 
$G$ has a filtration
$$0 = G_{d+1} \subset G_d \subset \cdots \subset G_1 \subset G_0 = G$$ such that each submodule $G_k$ and each quotient 
$G_k/G_{k+1}\cong\oplus_{\height(\q)=k}G_{(\q)}$
is Gorenstein injective over $S$ and 
each module 
$G_{(\q)}\cong \Gamma_\q(G_\q)$
is Gorenstein injective over $S$ and satisfies $\tp[q]$.
It follows that each submodule $G_k$,
each quotient $G_k/G_{k+1}$,
and each summand $G_{(\q)}$ is $C$-Gorenstein injective over $R$.

Note that the direct sum $0\oplus C$ is an ideal of $S$ with $(0\oplus C)^2=0$. 
It follows that each prime $\q\in\spec(S)$ is of the form $\q=\p\oplus C$ for a unique prime $\p\in\spec(R)$.
Thus, given an $R$-module $M$, one has $M_\q\cong M_\p$ and $\Gamma_\q(M)=\Gamma_\p(M)$;
this implies that
$G_{(\q)}\cong \Gamma_\q(G_\q)\cong \Gamma_\p(G_\p)\cong G_{(\p)}$.
Also, from~\cite[Proposition 3.6]{sather:scc}
we have $\q\in\supp_S(G)$ if and only if $\p\in\supp_R(G)$.
Thus, we have the desired filtration of $G$ over $R$.
The uniqueness and functoriality 
also follow from Theorem~\ref{20141010.4}; alternately, apply
the proof of~\cite[Theorem 3.1]{enochs}.
\end{proof}

Our proof of Theorem~\ref{20141114.2} from the introduction uses the following lemma.

\begin{lemma}\label{lem141218a}
Let $R$ be Cohen-Macaulay with a dualizing module $D$, and let  $\p$ be a minimal prime of $R$.
Assume that $C_\p\cong D_\p$.
Then each $R_{\p}$-module $T$  is $C$-Gorenstein injective over $R$.
\end{lemma}

\begin{proof}
Claim 1: The module $T$ is $D_{\p}$-Gorenstein injective over $R_{\p}$.
By definition, the Bass class $\catb{R_{\p}}{R_{\p}}=\catb{\hom_{R_{\p}}(D_{\p},D_{\p})}{R_{\p}}$ contains all $R$-modules, including $T$.
From Fact~\ref{fact141223c}, this implies that $T$ has finite $D_{\p}$-Gorenstein injective dimension over $R_{\p}$.
Since $R_{\p}$ is artinian, this implies that $D_{\p}\text{-}\gid_{R_{\p}}(T)\leq\dim(R)=0$ by Lemma~\ref{lem141223a}\eqref{lem141223a1}.
This establishes Claim 1.

Claim 2: One has $T\in\catb{\hom_R(C,D)}{R}$.
By definition again, we have $T\in\catb{R_{\p}}{R_{\p}}$.
From the isomorphism $C_{\p}\cong D_{\p}$, since $C_{\p}$ is semidualizing over $R_{\p}$,
there are isomorphisms
$$R_{\p}\cong\hom_{R_{\p}}(C_{\p},C_{\p})\cong\hom_{R_{\p}}(C_{\p},D_{\p})\cong\hom_{R}(C,D)_{\p}$$
so $T\in\catb{\hom_{R}(C,D)_{\p}}{R_{\p}}$.
Thus, we have $T\in\catb{\hom_R(C,D)}{R}$ by~\cite[Proposition 5.3(b)]{christensen:scatac}.
This establishes Claim 2.

Now we conclude our proof. 
Claim 2 implies that $T$ is in $\catb{\hom_R(C,D)}{R}$, so $C\text{-}\gid_R(T)<\infty$ by Fact~\ref{fact141223c}.
Because of the isomorphism $C_{\p}\cong D_{\p}$,
Claim 1 tells us that $T$ is $C_{\p}$-Gorenstein injective over $R_{\p}$.
Consider the left half of a complete $I_{C_{\p}}I$-resolution of $T$ over $R_{\p}$.
Truncating this yields an exact sequence
$$0\to V\to\hom_{R_{\p}}(C_{\p},I_{d-1})\to\cdots\to\hom_{R_{\p}}(C_{\p},I_0)\to T\to 0$$
where $d=\dim(R)$ and each module $I_j$ is injective over $R_{\p}$.
(If $d=0$, then we have $V=T$.)
Hom-tensor adjointness implies that this sequence has the following form:
$$0\to V\to\hom_{R}(C,I_{d-1})\to\cdots\to\hom_{R}(C,I_0)\to T\to 0.$$
Since $R_{\p}$ is flat over $R$ and $I_j$ is injective over $R_{\p}$, each $I_j$ is also injective over $R$.
Example~\ref{ex141223a} implies that each module $\hom_R(C,I_j)$ is $C$-Gorenstein injective over $R$.
We conclude that $T$ is $C$-Gorenstein injective over $R$, by
Lemma~\ref{lem141223a}\eqref{lem141223a2}.
\end{proof}

The next result shows that certain classes of $C$-Gorenstein injective $R$-modules are closed under tensor products in our setting. 
Recall that $C$ is \emph{generically dualizing} if the localization $C_{\p}$ is dualizing over $R_{\p}$ for each minimal prime $\p$ of $R$.

\begin{theorem}
\label{thm141218a}
Let $R$ be Cohen-Macaulay with a dualizing module $D$, and let $C$ be a semidualizing $R$-module. 
Assume that $C$ is generically dualizing.
If $G$ and $H$ are $C$-Gorenstein injective $R$-modules, then $G \otimes_R H$ is also $C$-Gorenstein injective.
\end{theorem}

\begin{proof}
Using Lemma~\ref{lem141224c} and Theorem~\ref{cor141225a} as in the proof of~\cite[Theorem 4.1]{enochs}, 
we assume without loss of generality that there is a  prime $\p\in\spec(R)$ such that $G$ and $H$ both satisfy $t(\p)$.
Moreover, by Lemma~\ref{20141002.2}, we assume without loss of generality that $\p$ is a minimal prime of $R$.
It follows that $G$ and $H$ are $R_{\p}$-modules.
By assumption, we have $C_{\p}\cong D_{\p}$.
Thus, Lemma~\ref{lem141218a} implies that $G\otimes_RH$ is $C$-Gorenstein injective over $R$, as desired.
\end{proof}

Here is Theorem~\ref{20141114.2} from the introduction.

\begin{theorem}
\label{cor150430a}
Let $R$ be Cohen-Macaulay with a dualizing module $D$. 
Assume that $R$ is generically Gorenstein.
If $G$ and $H$ are Gorenstein injective $R$-modules, then $G \otimes_R H$ is also Gorenstein injective.
\end{theorem}

\begin{proof}
This is the special case $C=R$ of Theorem~\ref{thm141218a}.
\end{proof}

\begin{rmk}\label{rmk141219a}
The preceding two results generalize~\cite[Theorem 4.1]{enochs}. However, our results are fundamentally different, as follows.
In~\cite[Theorem 4.1]{enochs}, the ring $R$ is Gorenstein, so every $R$-module has finite Gorenstein injective dimension. 
Hence, the point of~\cite[Theorem 4.1]{enochs} is not that the tensor product $G\otimes_RH$ has 
finite Gorenstein injective dimension, but that it has Gorenstein injective dimension 0.
On the other hand, in our setting, not every module has finite ($C$-)Gorenstein injective dimension. (See, e.g., the Example~\ref{ex141218a}.)
\end{rmk}

\begin{rmk}\label{rmk141231a}
It is natural to ask whether the results of this paper hold in a more general setting, e.g., if $R$ is only assumed to have a dualizing complex.
However, without some assumptions, the methods of proof in this paper break down quickly.
For instance, over the ring $R:=k[\![X,Y]\!]/(X^2,XY)$, the injective module $E:=E_R(k)$ satisfies $E\otimes_RE\cong k\neq 0$;
contrast this with Lemma~\ref{20141002.2}.
\end{rmk}

\bibliographystyle{amsplain}
\providecommand{\bysame}{\leavevmode\hbox to3em{\hrulefill}\thinspace}
\providecommand{\MR}{\relax\ifhmode\unskip\space\fi MR }
\providecommand{\MRhref}[2]{%
  \href{http://www.ams.org/mathscinet-getitem?mr=#1}{#2}
}
\providecommand{\href}[2]{#2}

\end{document}